\documentclass[a4paper,oneside,9pt]{article}%
\usepackage{makeidx}
\usepackage[english]{babel}
\usepackage{amsmath}
\usepackage{amsfonts}
\usepackage{amssymb}
\usepackage{stmaryrd}
\usepackage{graphicx}
\usepackage{mathrsfs}
\usepackage[colorlinks,linkcolor=red,anchorcolor=blue,citecolor=blue,urlcolor=blue]{hyperref}
\usepackage[symbol*,ragged]{footmisc}
\usepackage{bbm}
\usepackage{CJK}
\usepackage{indentfirst}
\usepackage{cases}

\allowdisplaybreaks[4]
\providecommand{\U}[1]{\protect\rule{.1in}{.1in}}
\providecommand{\U}[1]{\protect \rule{.1in}{.1in}}

\newtheorem{theorem}{Theorem}[section]

\newtheorem{corollary}[theorem]{Corollary}

\newtheorem{lemma}[theorem]{Lemma}

\newtheorem{remark}[theorem]{Remark}

\newenvironment{proof}[1][Proof]{\noindent \textbf{#1.} }{\  \rule{0.5em}{0.5em}}
\numberwithin{equation}{section}

\begin{document}

\title{On a multiplicative hybrid problem over almost-primes }

\author{Yuetong Zhao \,\,\,\,\&\,\,\,\,Wenguang Zhai
                    \vspace*{-4mm} \\
                    $\textrm{\small School of Mathematical Sciences, Beihang University}$
                     \vspace*{-4mm} \\
     \small  Beijing 100191, P. R. China
                     \vspace*{-4mm}  \\
     $\textrm{\small Department of Mathematics, China University of Mining and Technology}$
                    \vspace*{-4mm} \\
     \small  Beijing 100083, P. R. China
                      \vspace*{-4mm}\\
     }

\footnotetext
   {*Yuetong Zhao is the corresponding author.\\
     \textit{ E-mail addresses}:
     \href{mailto:yuetong.zhao.math@gmail.com}{yuetong.zhao.math@gmail.com} (Y. T. Zhao),
     \href{mailto:zhaiwg@hotmail.com}{zhaiwg@hotmail.com} (W. G. Zhai).\\
      }

\date{}
\maketitle

{\textbf{Abstract}}: Let $N$ be a large enough natural number, $\mathfrak{A}$ and $\mathfrak{B}$ be subsets of $\{N+1, \cdots , 2N\}$. In this paper, we prove that there exists integers $a, b$ with $a\in\mathfrak{A}$, $b\in\mathfrak{B}$ such that
\begin{equation*}
ab=P_k^2 + O(P_k^{1-\delta}),
\end{equation*}
where $0<\delta<\frac{1}{2}$ and $P_k$ denotes an almost-prime with at most $k$ prime factors, counted with multiplicity.


{\textbf{MR(2020) Subject Classification}}: 11N36, 11L07

\section{Introduction and main result}
Assume that $N$ is a sufficiently large integer, let $\mathfrak{S}=\mathfrak{S}(N)=\{N+1,\dots,2N\}$.  It is well-known that the sequence $(\sqrt{n})$, $n=1,2,\dots$, is uniformly distributed modulo one, which implies that the quantity $\min\limits_{n\in\mathfrak{S},c\in\mathbb{Z}}|\sqrt{n}-c|=o(1)$. More precisely, we can calculate that $\min\limits_{n\in\mathfrak{S},c\in\mathbb{Z}}|\sqrt{n}-c|\ll N^{-1/2}$. Now suppose $\mathfrak{A}$ is a random subset of $\mathfrak{S}$ with $|\mathfrak{A}|\gg N$, we want to know how large is the quantity $\min\limits_{n\in\mathfrak{A},c\in\mathbb{Z}}|\sqrt{n}-c|$? Is it possible that $\min\limits_{n\in\mathfrak{A},c\in\mathbb{Z}}|\sqrt{n}-c|=o(1)$? It is easy to prove that we can choose a subset $\mathfrak{A}$ of $\mathfrak{S}$ with $|\mathfrak{A}|\gg N$ such that $\min\limits_{n\in\mathfrak{A},c\in\mathbb{Z}}|\sqrt{n}-c|\gg 1$. This fact implies that there is no integer $n\in \mathfrak{A}$ such that $n$ is ``near a square".

However, the situation maybe different when there are two random subsets in $\mathfrak{S}$. In 1987, Iwaniec and S$\acute{a}$rk\"ozy \cite{IS} studied a multiplicative hybrid problem.
Let $\mathfrak{A},\mathfrak{B}$ be subsets of $\mathfrak{S}$ and $|\mathfrak{A}|\gg N$, $|\mathfrak{B}|\gg N$. They proved that there exist $a\in\mathfrak{A}$, $b\in\mathfrak{B}$ and an integer $c$ such that
\begin{equation}\label{ab}
  ab=c^2+O(c^{\frac{1}{2}}\log^{\frac{1}{2}}c),
\end{equation}
which implies that
\begin{equation*}
\min_{\substack{a\in \mathfrak{A},\,b\in\mathfrak{B}\\c\in\mathbb{Z}}}|\sqrt{ab}-c|\ll N^{-1/2}\log^{1/2}N,
\end{equation*}
 namely, we can find $a\in\mathfrak{A}$, $b\in\mathfrak{B}$ such that $ab$ is very ``near a square".

Subsequently, in 1995, Zhai \cite{Zhai} generalized this result to $k$ subsets of $\mathfrak{S}$. Assume that $G_1,\dots,G_k$ are $k$ subsets of $\mathfrak{S}$ and $|G_1|\gg N,\dots,|G_k|\gg N$. He proved that there exist integers $n_1,\dots, n_k$, $b$ with $n_1\in G_1,\dots, n_k\in G_k$ such that
\begin{equation}\label{k}
  n_1\cdots n_k=b^k+O(b^{k-3/2}\log^{(k-2)/2}b),\qquad k\geq 3.
\end{equation}
After that, Bordell$\grave{e}$s \cite{B} proved that there exist integers $n_1,\dots, n_k$, $b$ with $n_1\in G_1,\dots, n_k\in G_k$ such that
\begin{equation*}
  n_1\cdots n_k=b^k+O(b^{k-5/3+r(k)}), \qquad k\geq 2,
\end{equation*}
where $r(k)=2(9k+7)/(3(9k^2-3k+10))$. This result improved on (\ref{k}) for $k\geq 5$.

 According to (\ref{ab}), it is natural to conjecture that there exist integers $a$, $b$ with $a\in\mathfrak{A}$, $b\in\mathfrak{B}$ such that
\begin{equation*}
  ab=p^2+O(p^\theta),
\end{equation*}
where $p$ is prime and $\frac{1}{2}\leq \theta< 1$. In this paper, we obtain an approximate form of our conjecture.

For two finite random subsets $\mathfrak{A},\mathfrak{B}$ of  $\mathfrak{S}$ and a real number $\Delta$ with $0<\Delta\leq\frac{1}{2}$, we define
\begin{equation}\label{H1}
  H(\mathfrak{A},\mathfrak{B};\Delta)=\big|\big\{a\in\mathfrak{A},\, b\in\mathfrak{B} : \big\|\sqrt{ab}\big\|<\Delta \big\}\big|
\end{equation}
and
\begin{equation}\label{A}
  \mathscr{A}=\big\{l: l=l\big(\sqrt{ab}\big),\, a\in\mathfrak{A},\, b\in\mathfrak{B},\, \big\|\sqrt{ab}\big\|<\Delta\big\},
\end{equation}
where $l(t)$ denotes the function of the closest integer to $t$ and $\|t\|=\min\big\{\{t\}, 1-\{t\}\big\}$. We write
\begin{equation*}
  H (\mathscr{A};k)=|\{l\in \mathscr{A}: \Omega(l)\leq k\}|,
\end{equation*}
where $\Omega(l)$ denotes the number of prime factors of $l$, counted according to multiplicity. Now, we state our main results in the following.

\begin{theorem}\label{thm3}
Let $\mathfrak{A},\mathfrak{B}\subset\{N+1, \cdots, 2N\}$ and $0<\Delta\leq 1/2$. Moreover, we suppose that $\varepsilon>0$ is sufficiently small and $|\mathfrak{A}||\mathfrak{B}|\gg N^{4/3+\varepsilon}$, then we have
 \begin{equation*}
   H(\mathfrak{A},\mathfrak{B};\Delta)=2\Delta |\mathfrak{A}||\mathfrak{B}|+O\left(N(|\mathfrak{A}||\mathfrak{B}|)^{\frac{1}{4}}\log^{\frac{3}{2}}N\right).
 \end{equation*}
\end{theorem}


\begin{theorem}\label{thm1}
  Suppose that $\varepsilon>0$ is sufficiently small, $\mathfrak{A},\mathfrak{B}\subset\{N+1, \cdots, 2N\}$, $|\mathfrak{A}|\asymp N^{\eta}$, $|\mathfrak{B}|\asymp N^{\beta}$, $0<\eta\leq1$, $0<\beta\leq1$ and $\eta+\beta\geq 4(1+\delta)/3+\varepsilon$. Let $\Delta=N^{-\delta}$, $0<\delta<\frac{1}{2}$ and
  \begin{equation*}
  k=\left\lfloor\frac{2}{(\eta+\beta)/2-2/3-2\delta/3}\right\rfloor.
  \end{equation*}
  Then we have
  \begin{equation*}
    H(\mathscr{A};k)\geq C(\eta,\beta,\delta)\frac{\Delta|\mathfrak{A}||\mathfrak{B}|}{\log N}(1+o(1)),
  \end{equation*}
  where $C(\eta,\beta,\delta)>0$ is a constant depends on $\eta$, $\beta$, $\delta$. 
\end{theorem}

As an application of Theorem \ref{thm1} we get the following:
\begin{corollary}\label{1.3}
 Suppose that $\varepsilon>0$ is sufficiently small, $\mathfrak{A},\mathfrak{B}\subset\{N+1, \cdots, 2N\}$, $|\mathfrak{A}|\asymp N^{\eta}$, $|\mathfrak{B}|\asymp N^{\beta}$, $0<\eta\leq1$, $0<\beta\leq1$ and $\eta+\beta\geq 4(1+\delta)/3+\varepsilon$. Then there exist integers $a$, $b$ with $a\in \mathfrak{A}$, $b\in \mathfrak{B}$ and an integer having at most $k$ prime factors such that
\begin{equation}\label{Pk}
ab=P_k^2+O(P_k^{1-\delta}),
\end{equation}
where $0<\delta<\frac{1}{2}$ and
\begin{equation*}
k=\left\lfloor\frac{2}{(\eta+\beta)/2-2/3-2\delta/3}\right\rfloor.
\end{equation*}
Particularly, if the value of $\eta$ and $\beta$ are close to 1, then the formula (\ref{Pk}) holds for $0<\delta<1/14$ and $k=6$, which is the minimum value of $k$. In general cases, for different $k,\eta,\beta$, we can find the range of $\delta$:
\begin{equation*}
3(\eta+\beta)/4-1-3/k\leq\delta<3(\eta+\beta)/4-1-3/(k+1).
\end{equation*}

\end{corollary}

\begin{remark}
In the proof of Theorem \ref{thm1}, we directly relate $H(\mathscr{A};k)$ to sifting function. Although the
 minimum value of $k$ can only be 6, we get (\ref{Pk}) holds for each $0<\delta<\frac{1}{2}$ with different $k,\eta,\beta$; We use the weighted sieve of Chapter 9 of Pan and Pan \cite{PP} in the proof of Theorem \ref{thm2} below. By this way, $(\ref{Pk})$ can be derived for $k=4,5$.
\end{remark}

\begin{theorem}\label{thm2}
Suppose that $\varepsilon>0$ is sufficiently small, $\mathfrak{A},\mathfrak{B}\subset\{N+1, \cdots, 2N\}$, $|\mathfrak{A}|\gg N^{1-\varepsilon}$, $|\mathfrak{B}|\gg N^{1-\varepsilon}$. Let $\Delta=N^{-\delta}$, $0<\delta<\frac{1}{2}$ and $k=4,5$.
  Then we have the following results:\\
(1) If $k=5$, $0<\delta<\frac{1}{10}$, we have
  \begin{equation*}
    H(\mathscr{A};5)\geq C(\delta,5)\frac{\Delta|\mathfrak{A}||\mathfrak{B}|}{\log N}(1+o(1)),
  \end{equation*}
    where $C(\delta,5)>0$ is a constant depends on $\delta$;\\
(2) If $k=4$, $0<\delta<\frac{121}{10000}$, we have
  \begin{equation*}
    H(\mathscr{A};4)\geq C(\delta,4)\frac{\Delta|\mathfrak{A}||\mathfrak{B}|}{\log N}(1+o(1)),
  \end{equation*}
 where $ C(\delta,4)>0.0023205$ is a constant depends on $\delta$.
\end{theorem}

\smallskip
\textbf{Notation.}
Throughout this paper,  $\varepsilon$ always denotes an arbitrary small positive constant, which may not be the same at different occurrences. The symbol $N$ always denotes a sufficiently large natural number. For any real number $t$, $\lfloor t\rfloor$ denotes its integer part, $\{t\}$ its fractional part, $\psi(t)=t-\lfloor t\rfloor-1/2$ and $\|t\|=\min\big\{\{t\}, 1-\{t\}\big\}$. As usual, $e(t)=e^{2\pi it}$. Let $p$, with or without subscripts, always denote a prime number. The notation $f(x)\ll g(x)$ means that $f(x)=O(g(x))$. If we have simultaneously $f(x)\ll g(x)$ and $g(x)\ll f(x)$, then we shall write $f(x)\asymp g(x)$. The symbol $P_k$ denotes an almost-prime with at most $k$ prime factors, counted by multiplicity. We write $|A|$ for the cardinality of set $A$. The function $\Omega(n)$ denotes the number of prime factors of $n$, counted according to multiplicity. Let $\mu(n)$ denote M\"obius' function. We write $\tau (n)$ to denote divisor function. The symbol $\mathbb{Z}$ denotes the set of all integers. For any set $G$ of real numbers, we denote by $\chi_G$ the characteristic function of $G$, which means
\begin{equation*}
 \chi_{G}(n)= \left\{
      \begin{array}{cll}
      1,  &&  n\in G,\\
      0,  &&  n\notin G.
      \end{array}
   \right.
\end{equation*}

\section{Some Lemmas}
In this section, we shall give some preliminary lemmas, which are necessary in the proof of the main result.
\begin{lemma}\label{lem1}
Suppose
\begin{equation*}
  \mathscr{B}_{\varphi\psi}(\mathscr{X},\mathscr{Y})=\sum_{r}\sum_{s}\varphi_{r}\psi_{s}e(x_ry_s)
\end{equation*}
where $\mathscr{X}=(x_r)$, $\mathscr{Y}=(y_s)$ are finite sequences of real numbers with $|x_r|\leq X$, $|y_s|\leq Y$, and $\varphi_r$, $\psi_s$ are complex numbers. Then, we have
\begin{equation*}
  |\mathscr{B}_{\varphi\psi}(\mathscr{X},\mathscr{Y})|^2\leq 20(1+XY)\mathscr{B}_{\varphi}(\mathscr{X},Y)\mathscr{B}_{\psi}(\mathscr{Y},X)
\end{equation*}
with
\begin{equation*}
  \mathscr{B}_{\varphi}(\mathscr{X},Y)=\sum_{|x_{r_1}-x_{r_2}|\leq Y^{-1}}|\varphi_{r_1}\varphi_{r_2}|
\end{equation*}
and $\mathscr{B}_{\psi}(\mathscr{Y},X)$ defined similarly.
\end{lemma}
\begin{proof}
 See Proposition 1 of Fouvry and Iwaniec \cite{FI}.  $\hfill$
\end{proof}

\begin{lemma}\label{lem2}
Let $\alpha\beta\neq0$, $\Theta>0$, $M\geq1$, $N\geq1$, and $\mathscr{S}(M,N;\Theta)$ be the number of quadruples (m,$\tilde{m}$,n,$\tilde{n}$) such that
\begin{equation*}
 \left|\left(\frac{\tilde{m}}{m}\right)^{\alpha}-\left(\frac{\tilde{n}}{n}\right)^{\beta}\right|<\Theta,
\end{equation*}
with $M\leq m, \tilde{m}<2M$ and $N\leq n, \tilde{n}<2N$. Then, we have
\begin{equation*}
  \mathscr{S}(M,N;\Theta)\ll MN\log2MN+\Theta M^2N^2.
\end{equation*}
\end{lemma}
\begin{proof}
 See Lemma 1 of Fouvry and Iwaniec \cite{FI}.  $\hfill$
\end{proof}

\begin{lemma}\label{lem2'}
Let $\mathfrak{B}\subset\{N+1,\cdots,2N\}$, $X\geq 1$ and $\mathscr{Q}(\mathfrak{B};X)$ be the number of two-tuples $(b,b_1)$ such that
\begin{equation*}
 \left|\sqrt{b}-\sqrt{b_1}\right|<(2X)^{-1},
\end{equation*}
with $ b, b_1\in \mathfrak{B}$. Then, we have
\begin{equation*}
  \mathscr{Q}(\mathfrak{B};X)\leq (1+2\sqrt{2N}X^{-1})|\mathfrak{B}|.
\end{equation*}
\end{lemma}
\begin{proof}
 See Lemma 2 of Iwaniec and S$\acute{a}$rk\"ozy \cite{IS}.  $\hfill$
\end{proof}

\begin{lemma}\label{lem3}
 Suppose that $F(u)$ and $f(u)$ are continuous functions, which satisfy the following equation
 \begin{equation}\label{fF}
 \left\{
      \begin{array}{cll}
      F(u)=\frac{2e^{\gamma}}{u},\quad f(u)=0, &&   0< u\leq 2,\\
      (uF(u))'=f(u-1),\quad (uf(u))'=F(u-1), &&  u>2,
      \end{array}
   \right.
\end{equation}
where $\gamma$ is Euler's constant.
Then we have
\begin{equation*}
  F(u)=\frac{2e^{\gamma}}{u},\qquad 0< u\leq 3,
\end{equation*}
\begin{equation*}
  F(u)=\frac{2e^{\gamma}}{u}\left(1+\int_{2}^{u-1}\frac{\log(t-1)}{t}dt\right),\qquad 3\leq u\leq 5,
\end{equation*}
\begin{equation*}
  f(u)=\frac{2e^{\gamma}}{u}\left(\log (u-1)+\int_{3}^{u-1}\frac{1}{t}\left(\int_{2}^{t-1}\frac{\log(s-1)}{s}ds\right)dt\right),\qquad 4\leq u\leq 6.
\end{equation*}
 \end{lemma}
\begin{proof}
See Section 2 of Chapter 8  of Halberstam and Richert \cite{HR}.  $\hfill$
\end{proof}

\begin{lemma}\label{lem3'}
Let $F(u)$ and $f(u)$ be continuous functions satisfying (\ref{fF}). Then the following are valid:\\
(1) $F(u)$ is a positive and strictly decreasing function, and satisfies
\begin{equation*}
F(u)=1+O(e^{-u}),
\end{equation*}
(2) $f(u)$ is a positive and strictly increasing function on $u\geq 2$, and satisfies
\begin{equation*}
f(u)=1+O(e^{-u}),
\end{equation*}
(3) $F(u)-f(u)>0$, $u\geq 1$.
 \end{lemma}
\begin{proof}
See Section 2 of Chapter 8 of Halberstam and Richert \cite{HR}.  $\hfill$
\end{proof}

\begin{lemma}\label{lem4}
If $f$ is continuously differentiable, $f'$ is monotonic, and $||f'||\geq \lambda>0$ on $I$ then
\begin{equation*}
\sum_{n\in I}e(f(n))\ll \lambda^{-1}.
\end{equation*}
\end{lemma}
\begin{proof}
 See Theorem 2.1 of Graham and Kolesnik \cite{GK}.  $\hfill$
\end{proof}

\begin{lemma}\label{lem5}
For any $H\geq 2$, we have
\begin{equation*}
  \psi(t)=\sum_{0<|h|\leq H}u(h)e(ht)+O\left(\sum_{|h|\leq H}v(h)e(ht)\right),
\end{equation*}
where
\begin{equation}\label{condition}
  u(h)\ll \frac{1}{|h|},\quad v(h)\ll\frac{1}{H},\quad \sum_{|h|\leq H}v(h)e(ht)\geq 0.
\end{equation}
\end{lemma}
\begin{proof}
 See Lemma 1 of Rivat and S$\acute{a}$rk\"ozy \cite{RS}.  $\hfill$
\end{proof}

For given $z\geq 2$, we define the sifting function
\begin{equation}\label{sift}
  S(\mathscr{A},z)=\big|\{l\in\mathscr{A}: (l, P(z))=1\}\big|,
\end{equation}
where
\begin{equation*}
  P(z)=\prod_{\substack{p<z\\p\in\mathscr{P}}}p
\end{equation*}
and $\mathscr{P}$ is the set of all primes. For $d|P(z)$, we define
\begin{equation}\label{Ad}
\mathscr{A}_d=\{l\in\mathscr{A}: d|l\}.
\end{equation}
Assume that $|\mathscr{A}_d|$ can be written in the form
\begin{equation*}
|\mathscr{A}_d|=\frac{\omega(d)}{d}X+r(\mathscr{A},d),
\end{equation*}
where $X>1$ and $\omega(d)$ is a non-negative and multiplicative function, such that $\frac{\omega(d)}{d}X$ approximates $|\mathscr{A}_d|$ closely, that is, the error term $r(\mathscr{A},d)$ is small on average over d. Below we state the Lemmas of classical linear sieve method.

\begin{lemma}\label{sieve}
 Suppose
 \begin{equation}\label{V}
  V(z)=\prod_{p<z}\left(1-\frac{1}{p}\right)=\frac{e^{-\gamma}}{\log z}\left(1+O\left(\frac{1}{\log z}\right)\right),
\end{equation}
where $\gamma$ is Euler's constant. We have the conditions:\\
(1) For some suitable constant $A_1\geq 1$, we have
\begin{equation*}
 (\Omega_1):\qquad 0\leq \frac{\omega(p)}{p}\leq 1-\frac{1}{A_1}.
\end{equation*}
(2) Assume that $2\leq z\leq y$, the constant $A_2\geq 1$ and  $L\geq 1$ are independent of $z$ and $y$. We have
\begin{equation*}
(\Omega_2(1,L)): \qquad -L \leq \sum_{z\leq p<y}\frac{\omega(p)\log p}{p}-\log\left(\frac{y}{z}\right)\leq A_2.
\end{equation*}
(3) There exist constant $0<\alpha\leq 1$, $A_3\geq 1$ and $A_4\geq 1$ such that
\begin{equation*}
(R(1,\alpha)): \qquad \sum_{
d<X^{\alpha}/(\log X)^{A_3} \\ }\mu^2(d)3^{\nu(d)}|r(\mathscr{A}, d)|\leq A_4\frac{X}{\log^2 X}, \quad X\geq 2,
\end{equation*}
where $\nu(d)$ denotes the number of different prime factors of $d$.\\
If these three conditions hold, then for $z\leq X$ we have
\begin{equation}\label{upper}
S(\mathscr{A}, z)\leq XV(z)\left\{F\left(\alpha\frac{\log X}{\log z}\right)+C\frac{L}{(\log X)^{1/14}}\right\}
\end{equation}
and
\begin{equation}\label{lower}
S(\mathscr{A}, z)\geq XV(z)\left\{f\left(\alpha\frac{\log X}{\log z}\right)-C\frac{L}{(\log X)^{1/14}}\right\},
\end{equation}
where $C>0$ is an absolute constant, $F$ and $f$ are the functions of Lemma \ref{lem3}.
\end{lemma}
\begin{proof}
 See Theorem 8.4 of Halberstam and Richert \cite{HR}.  $\hfill$
\end{proof}

\begin{lemma}\label{sieve'}
Suppose that the conditions $(\Omega_1)$ and $(\Omega_2(1,L))$ hold. Then, for $2\leq z\leq \xi$, we have
\begin{equation}\label{upperp}
S(\mathscr{A}_p, z)\leq \frac{\omega(p)}{p}XV(z)\left\{F\left(\frac{\log\xi^2}{\log z}\right)+C'\frac{L}{(\log \xi)^{1/14}}\right\}+\sum_{\substack{n<\xi^2\\n|P(z)}}3^{\nu(n)}|r(\mathscr{A},pn)|
\end{equation}
and
\begin{equation}\label{lowerp}
S(\mathscr{A}_p, z)\geq \frac{\omega(p)}{p}XV(z)\left\{f\left(\frac{\log \xi^2}{\log z}\right)-C'\frac{L}{(\log \xi)^{1/14}}\right\}-\sum_{\substack{n<\xi^2\\n|P(z)}}3^{\nu(n)}|r(\mathscr{A},pn)|,
\end{equation}
where $C'>0$ is an absolute constant.
\end{lemma}
\begin{proof}
 See Theorem 8.3 of Halberstam and Richert \cite{HR}.  $\hfill$
\end{proof}

\section{Proof of Theorem \ref{thm3}  and Estimate of $r(\mathscr{A},d)$}

We know that
 \begin{equation}\label{Re}
   \lfloor\sqrt{ab}+\Delta\rfloor-\lfloor\sqrt{ab}-\Delta\rfloor= \left\{
      \begin{array}{cll}
      1,  &&  \text{if}\,\,\big\|\sqrt{ab}\big\|<\Delta,\\
      0,  &&  \text{otherwise},
      \end{array}
   \right.
 \end{equation}
 in fact
 \begin{align*}
   \big\|\sqrt{ab}\big\|<\Delta & \iff \big|\sqrt{ab}-l\big|<\Delta \iff \sqrt{ab}-\Delta<l<\sqrt{ab}+\Delta\nonumber\\
    & \iff \lfloor\sqrt{ab}+\Delta\rfloor-\lfloor\sqrt{ab}-\Delta\rfloor=1.
 \end{align*}
Thus, recalling the definition (\ref{H1}), we have
\begin{align}\label{Hab}
  H(\mathfrak{A},\mathfrak{B};\Delta) & =\sum_{a\in\mathfrak{A}}\sum_{b\in\mathfrak{B}}\left(\lfloor\sqrt{ab}+\Delta\rfloor-\lfloor\sqrt{ab}-\Delta\rfloor\right) \nonumber\\
    & =2\Delta |\mathfrak{A}||\mathfrak{B}|+\sum_{a\in\mathfrak{A}}\sum_{b\in\mathfrak{B}}\left(\psi\big(\sqrt{ab}-\Delta\big)-\psi\big(\sqrt{ab}+\Delta\big)\right).
\end{align}
Using Lemma \ref{lem5}, we can get
\begin{align}\label{sumpsi}
  &\sum_{a\in\mathfrak{A}}\sum_{b\in\mathfrak{B}}  \psi\big(\sqrt{ab}\pm\Delta\big) \nonumber\\
  =& \sum_{a\in\mathfrak{A}}\sum_{b\in\mathfrak{B}}\sum_{0< |h|\leq H'}u(h)e(h(\sqrt{ab}\pm\Delta))+O\left(\sum_{a\in\mathfrak{A}}\sum_{b\in\mathfrak{B}}\sum_{|h|\leq H'}v(h)e(h(\sqrt{ab}\pm\Delta))\right).
\end{align}
Since (\ref{condition}), we can estimate the main term of (\ref{sumpsi}) is
\begin{equation}\label{MTpsi}
  \ll \sum_{0< |h|\leq H'}\frac{1}{|h|}\left|\sum_{a\in\mathfrak{A}}\sum_{b\in\mathfrak{B}}e(h(\sqrt{ab}\pm\Delta))\right|.
\end{equation}
We take the term $h=0$ outside of the summation, by (\ref{condition}), the error term of (\ref{sumpsi}) is
\begin{equation}\label{ETpsi}
  \ll \frac{|\mathfrak{A}||\mathfrak{B}|}{H'}+\frac{1}{H'}\sum_{0< |h|\leq H'}\left|\sum_{a\in\mathfrak{A}}\sum_{b\in\mathfrak{B}}e(h(\sqrt{ab}\pm\Delta))\right|.
\end{equation}
By (\ref{sumpsi})--(\ref{ETpsi}), we get
\begin{equation*}
  \sum_{a\in\mathfrak{A}}\sum_{b\in\mathfrak{B}}  \psi\big(\sqrt{ab}\pm\Delta\big)\ll \frac{|\mathfrak{A}||\mathfrak{B}|}{H'}+\sum_{0< |h|\leq H'}\frac{1}{|h|}\left|\sum_{a\in\mathfrak{A}}\sum_{b\in\mathfrak{B}}e(h(\sqrt{ab}\pm\Delta))\right|.
\end{equation*}
For the second term on the right hand side of above formula, we divide the summation ranges of $h$ into a dyadic interval with $h\sim H_0$, $H_0=2^{-j-1}H'$, $0\leq j\ll \log H'$, which gives
\begin{align}\label{sumpsi0}
&\sum_{a\in\mathfrak{A}}\sum_{b\in\mathfrak{B}}  \psi\big(\sqrt{ab}\pm\Delta\big) \nonumber\\
\ll & \frac{|\mathfrak{A}||\mathfrak{B}|}{H'}+\frac{\log H'}{H_0}\sum_{ h\sim H_0}\left|\sum_{a\in\mathfrak{A}}\sum_{b\in\mathfrak{B}}e(h\sqrt{ab})\right|\nonumber\\
  = & \frac{|\mathfrak{A}||\mathfrak{B}|}{H'}+\frac{\log H'}{H_0}\sum_{ h\sim H_0}g(h)\sum_{a\in\mathfrak{A}}\sum_{b\in\mathfrak{B}}e(h\sqrt{ab})
\end{align}
for some $1\ll H_0\ll H^{\prime}$. Here $(g(h))$ is a complex sequence with $|g(h)|\leq 1$.\\

Next, our aim is to estimate the sum
\begin{equation*}
  \mathscr{L}:=\sum_{ h\sim H_0}g(h)\sum_{a\in\mathfrak{A}}\sum_{b\in\mathfrak{B}}e(h\sqrt{ab})
  =\sum_{ h\sim H_0}g(h)\sum_{ a\sim N}\sum_{ b\sim N}\chi_{\mathfrak{A}}(a)\chi_{\mathfrak{B}}(b)e(h\sqrt{ab}).
\end{equation*}
 Applying Lemma \ref{lem1} to the sequences $\mathscr{X}_0= \left\{h\sqrt{a}\,|\,h\sim H_0,\, a\in \mathfrak{A} \right\}$ and  $\mathscr{Y}_0=\big\{\sqrt{b}\,|\, b\in \mathfrak{B}\big\}$, we get
\begin{equation}\label{L}
 \mathscr{L}^2\ll H_0N\mathscr{E}_1\mathscr{E}_2,
\end{equation}
where
\begin{equation*}
  \mathscr{E}_1=\sum_{\left|h_1\sqrt{a_1}-h_2\sqrt{a_2}\right|\leq \frac{1}{\sqrt{N}}}\left|g(h_1)g(h_2)\chi_{\mathfrak{A}}(a_1)\chi_{\mathfrak{A}}(a_2)\right|
\end{equation*}
and
\begin{equation*}
  \mathscr{E}_2=\sum_{\left|\sqrt{b_1}-\sqrt{b_2}\right|\leq \frac{1}{H_0\sqrt{N}}}\left|\chi_{\mathfrak{B}}(b_1)\chi_{\mathfrak{B}}(b_2)\right|.
\end{equation*}
We estimate $\mathscr{E}_1$ and $\mathscr{E}_2$ by Lemma \ref{lem2} and Lemma \ref{lem2'} respectively. Then, we can get
\begin{equation}\label{E1}
  \mathscr{E}_1 \ll \sum_{\left|\frac{h_1}{h_2}-\frac{\sqrt{a_2}}{\sqrt{a_1}}\right|\leq \frac{1}{H_0N}}1
               \ll H_0N\log N
\end{equation}
and
\begin{equation}\label{E2}
  \mathscr{E}_2  \ll(1+H_0^{-1})|\mathfrak{B}|\ll|\mathfrak{B}|.
\end{equation}
Combining (\ref{L})--(\ref{E2}), we have
\begin{equation}\label{L1}
  \mathscr{L}  \ll  H_0N|\mathfrak{B}|^{\frac{1}{2}}\log^{\frac{1}{2}}N.
\end{equation}
 Similarly, by the same process, if we assume the sequences $\mathscr{X}_0= \left\{h\sqrt{b}\,|\,h\sim H_0,\, b\in \mathfrak{B} \right\}$ and  $\mathscr{Y}_0=\big\{\sqrt{a}\,|\, a\in \mathfrak{A}\big\}$, we can find that
 \begin{equation}\label{L2}
  \mathscr{L}  \ll  H_0N|\mathfrak{A}|^{\frac{1}{2}}\log^{\frac{1}{2}}N.
\end{equation}
 Thus
 \begin{equation}\label{LL}
  \mathscr{L}  \ll  H_0N(|\mathfrak{A}||\mathfrak{B}|)^{\frac{1}{4}}\log^{\frac{1}{2}}N.
\end{equation}
 By (\ref{sumpsi0}) and (\ref{LL}), we get
 \begin{equation}\label{sumpsi2}
   \sum_{a\in\mathfrak{A}}\sum_{b\in\mathfrak{B}}  \psi\big(\sqrt{ab}\pm\Delta\big)\ll N(|\mathfrak{A}||\mathfrak{B}|)^{\frac{1}{4}}\log^{\frac{3}{2}}N,
 \end{equation}
 here we choose $H'=N^{-1}(|\mathfrak{A}||\mathfrak{B}|)^{\frac{3}{4}}$.
 Theorem \ref{thm3} is derived from (\ref{Hab}) and (\ref{sumpsi2}).

The result of Theorem \ref{thm3} shows that there exists integers $a\in\mathfrak{A},b\in\mathfrak{B}$ such that  $ab$ is `` near a square". Then, we aim to prove that there exists integers $a\in\mathfrak{A}$, $b\in\mathfrak{B}$ such that $ab$ is `` near an almost-prime square", which means we get a refinement of above result. Before proving the main result of this paper in Section 4, we intend to give a vital Lemma first, which is one of keys in making the classical linear sieve method available.


Recalling the definition  (\ref{A}), (\ref{Ad}) and (\ref{Re}), we have
\begin{align}\label{Ad1}
  |\mathscr{A}_d| & =\sum_{\substack{l\in\mathscr{A}\\d|l}}1 =\sum_{\substack{|\sqrt{ab}-l|<\Delta\\a\in\mathfrak{A},b\in\mathfrak{B}\\d|l}}1
                    =\sum_{a\in\mathfrak{A}}\sum_{b\in\mathfrak{B}}\left(\left\lfloor\frac{\sqrt{ab}+\Delta}{d}\right\rfloor-\left\lfloor\frac{\sqrt{ab}-\Delta}{d}\right\rfloor\right)\nonumber\\
                  & =\frac{2\Delta|\mathfrak{A}||\mathfrak{B}|}{d}+
                  \sum_{a\in\mathfrak{A}}\sum_{b\in\mathfrak{B}}\left(\psi\left(\frac{\sqrt{ab}-\Delta}{d}\right)-\psi\left(\frac{\sqrt{ab}+\Delta}{d}\right)\right)\nonumber\\
                  & := \frac{\omega(d)}{d}X+r(\mathscr{A},d),
\end{align}
where
\begin{equation}\label{defX}
  \omega(d)=1, \qquad X=2\Delta|\mathfrak{A}||\mathfrak{B}|,
\end{equation}
\begin{equation}\label{defr(A,d)}
    r(\mathscr{A},d)=\sum_{a\in\mathfrak{A}}\sum_{b\in\mathfrak{B}}\left(\psi\left(\frac{\sqrt{ab}-\Delta}{d}\right)-\psi\left(\frac{\sqrt{ab}+\Delta}{d}\right)\right).
\end{equation}
Next, we give the following result:

\begin{lemma}\label{lem0}
 We define $X$ and $r(\mathscr{A},d)$ by (\ref{defX}) and (\ref{defr(A,d)}). Suppose that $0<\delta<\frac{1}{2}$, $\Delta=N^{-\delta}$, $|\mathfrak{A}|\asymp N^{\eta}$, $|\mathfrak{B}|\asymp N^{\beta}$, $0<\eta\leq1$, $0<\beta\leq1$, $\eta+\beta\geq 4(1+\delta)/3+\varepsilon$ and $\alpha=\frac{(\eta+\beta)/2-2/3-2\delta/3}{\eta+\beta-\delta}-\varepsilon$. Then, for each $1\leq d\leq X^{\alpha}$, we have
  \begin{equation*}
     r(\mathscr{A},d)\ll \frac{XN^{-\varepsilon}}{d}.
  \end{equation*}
\end{lemma}
\begin{proof}
By Lemma \ref{lem5}, we obtain
\begin{align}
  \sum_{a\in\mathfrak{A}}\sum_{b\in\mathfrak{B}}\psi\left(\frac{\sqrt{ab}\pm\Delta}{d}\right) &= \sum_{a\in\mathfrak{A}}\sum_{b\in\mathfrak{B}}\sum_{0<|h|\leq H}u(h)e\left(\frac{h(\sqrt{ab}\pm\Delta)}{d}\right)\nonumber\\
 &\quad +O\left(\sum_{a\in\mathfrak{A}}\sum_{b\in\mathfrak{B}} \sum_{|h|\leq H}v(h)e\left(\frac{h(\sqrt{ab}\pm\Delta)}{d} \right)\right).
\end{align}
Using a similar argument of getting (\ref{sumpsi0}), we have
\begin{align}\label{sumpsi1}
  &\sum_{a\in\mathfrak{A}}\sum_{b\in\mathfrak{B}}  \psi\left(\frac{\sqrt{ab}\pm\Delta}{d}\right) \nonumber\\
    \ll & \frac{|\mathfrak{A}||\mathfrak{B}|}{H}+\frac{\log H}{H_1}\sum_{ h\sim H_1}c(h)\sum_{a\in\mathfrak{A}}\sum_{b\in\mathfrak{B}}e\left(\frac{h\sqrt{ab}}{d}\right),
  \end{align}
where $(c(h))$ is a complex sequence with $|c(h)|\leq 1$ and $H_1$ satisfies
\begin{equation}\label{H1}
H_1=2^{-j-1}H, \quad 0\leq j\leq J\ll \log H.
\end{equation}
It remains to estimate
\begin{equation*}
  \mathscr{M}:=\sum_{ h\sim H_1}c(h)\sum_{a\in\mathfrak{A}}\sum_{b\in\mathfrak{B}}e\left(\frac{h\sqrt{ab}}{d}\right)=\sum_{h\sim H_1}c(h) \sum_{a\sim N}\chi_{\mathfrak{A}}(a)\sum_{b\sim N}\chi_{\mathfrak{B}}(b)e\left(\frac{h\sqrt{ab}}{d}\right).
\end{equation*}
Applying Lemma \ref{lem1} to the sequences $\mathscr{X}_1= \left\{\frac{h\sqrt{a}}{d}\,|\,h\sim H_1,\, a\in \mathfrak{A} \right\}$ and  $\mathscr{Y}_1=\big\{\sqrt{b}\,|\, b\in \mathfrak{B}\big\}$, we get
\begin{equation}\label{T1'1}
  \mathscr{M}^{2}\ll \frac{NH_1}{d}\mathscr{D}_1\mathscr{D}_2,
\end{equation}
where
\begin{equation*}
  \mathscr{D}_1=\sum_{\left|h_1\sqrt{a_1}-h_2\sqrt{a_2}\right|\leq \frac{d}{\sqrt{N}}}|c(h_1)c(h_2)\chi_{\mathfrak{A}}(a_1)\chi_{\mathfrak{A}}(a_2)|,
\end{equation*}
and
\begin{equation*}
  \mathscr{D}_2=\sum_{\left|\sqrt{b_1}-\sqrt{b_2}\right|\leq \frac{d}{H_1\sqrt{N}}}|\chi_{\mathfrak{B}}(b_1)\chi_{\mathfrak{B}}(b_2)|.
\end{equation*}
Bounding $\mathscr{D}_1$ and $\mathscr{D}_2$ by Lemma \ref{lem2} and Lemma \ref{lem2'} respectively, we get
\begin{equation}\label{D1}
  \mathscr{D}_1 \ll \sum_{\left|\frac{h_1}{h_2}-\frac{\sqrt{a_2}}{\sqrt{a_1}}\right|\leq \frac{d}{NH_1}}1
               \ll NH_1\log 2NH_1+dNH_1
\end{equation}
and
\begin{equation}\label{D2}
  \mathscr{D}_2  \ll(1+dH_1^{-1})|\mathfrak{B}|.
\end{equation}
Combining (\ref{T1'1})--(\ref{D2}), we have
\begin{equation*}
  \mathscr{M}  \ll  NH_1|\mathfrak{B}|^{\frac{1}{2}}\left(1+d^{\frac{1}{2}}H_1^{-\frac{1}{2}}\right)\log^{\frac{1}{2}}2NH_1.
\end{equation*}
Similarly, by the same process, if we assume the sequences $\mathscr{X}_1= \left\{\frac{h\sqrt{b}}{d}\,|\,h\sim H_1,\, b\in \mathfrak{B} \right\}$ and  $\mathscr{Y}_1=\big\{\sqrt{a}\,|\, a\in \mathfrak{A}\big\}$, we can find that
\begin{equation*}
  \mathscr{M}  \ll  NH_1|\mathfrak{A}|^{\frac{1}{2}}\left(1+d^{\frac{1}{2}}H_1^{-\frac{1}{2}}\right)\log^{\frac{1}{2}}2NH_1.
\end{equation*}
Therefore, we get
\begin{equation}\label{T1'2}
  \mathscr{M}  \ll  NH_1(|\mathfrak{A}||\mathfrak{B}|)^{\frac{1}{4}}\left(1+d^{\frac{1}{2}}H_1^{-\frac{1}{2}}\right)\log^{\frac{1}{2}}2NH_1.
\end{equation}
Inserting (\ref{T1'2}) into (\ref{sumpsi1}), we conclude that
  \begin{align}\label{sumpsi2}
    &\sum_{a\in\mathfrak{A}}\sum_{b\in\mathfrak{B}}  \psi\left(\frac{\sqrt{ab}\pm\Delta}{d}\right) \nonumber\\
      \ll & |\mathfrak{A}||\mathfrak{B}|H^{-1}+N(|\mathfrak{A}||\mathfrak{B}|)^{\frac{1}{4}}\left(1+d^{\frac{1}{2}}H_1^{-\frac{1}{2}}\right)(\log N)^{\frac{3}{2}}.
    \end{align}
    Take $H=d\Delta^{-1}N^{2\varepsilon}$. Since $H_1$ satisfies (\ref{H1}), we can calculate
\begin{equation*}
  N(|\mathfrak{A}||\mathfrak{B}|)^{\frac{1}{4}}(\log N)^{\frac{3}{2}}\ll d^{-1}XN^{-\varepsilon},
\end{equation*}
and
\begin{equation*}
  N(|\mathfrak{A}||\mathfrak{B}|)^{\frac{1}{4}}d^{\frac{1}{2}}H_1^{-\frac{1}{2}}(\log N)^{\frac{3}{2}}\ll d^{-1}XN^{-\varepsilon},
\end{equation*}
for each $1\leq d\leq X^{\alpha}$, where $\alpha=\frac{(\eta+\beta)/2-2/3-2\delta/3}{\eta+\beta-\delta}-\varepsilon$.
Hence
\begin{equation}\label{FT1}
  \sum_{a\in\mathfrak{A}}\sum_{b\in\mathfrak{B}}  \psi\left(\frac{\sqrt{ab}\pm\Delta}{d}\right)\ll d^{-1}XN^{-\varepsilon}.
\end{equation}
By (\ref{defr(A,d)}) and (\ref{FT1}), we complete the proof of Lemma \ref{lem0}.
  $\hfill$
\end{proof}

\section{Proof of Theorem \ref{thm1}  }

We directly relate $H(\mathscr{A};k)$ to sifting function $S(\mathscr{A},z)$. Obviously, we can see that
\begin{equation}\label{HAk1}
  H(\mathscr{A};k)\geq \sum_{\substack{l\in\mathscr{A}\\(l,P((3N)^{1/(k+1)}))=1}}1=S(\mathscr{A},(3N)^{\frac{1}{k+1}}).
\end{equation}
For applying Lemma \ref{sieve} to get the lower bound of $S(\mathscr{A}, (3N)^{\frac{1}{k+1}})$, we have to check the condition $(\Omega_1)$, $(\Omega_2(1,L))$ and $(R(1,\alpha))$. Due to $\omega(d)=1$, the condition $(\Omega_1)$ and $(\Omega_2(1,L))$ hold. By Lemma \ref{lem0}, we obtain
\begin{equation}\label{R}
\sum_{d<X^{\alpha}}\mu^2(d)3^{\nu(d)}|r(\mathscr{A}, d)|
\ll  N^{-\varepsilon}X\sum_{d<X^{\alpha}} \frac{\mu^2(d)3^{\nu(d)}}{d}\ll N^{-\varepsilon}X(\log X)^3\ll XN^{-\frac{\varepsilon}{2}}.
\end{equation}
Hence, the condition $(R(1,\alpha))$ holds. Then, by (\ref{V}) and (\ref{lower}), we get
\begin{align}\label{SAv}
   S(\mathscr{A}, (3N)^{\frac{1}{v}})&\geq XV((3N)^{\frac{1}{v}})(f(\alpha v(\eta+\beta-\delta) )+o(1))\nonumber\\
   &\geq \frac{X}{\log (3N)}\cdot ve^{-\gamma} (f\left(\alpha v(\eta+\beta-\delta)\right)+o(1)),
\end{align}
where $v>0$. According to Lemma \ref{lem3'}, we can conclude that
\begin{equation*}
f(u)=0,\quad u\leq 2\,;\qquad f(u)>0,\quad u>2.
\end{equation*}
We take $v=k+1$. Since $k=\left\lfloor\frac{2}{(\eta+\beta)/2-2/3-2\delta/3}\right\rfloor$, the condition
\begin{equation*}
\alpha(k+1)(\eta+\beta-\delta)> 2
\end{equation*}
 holds, i.e. $f(\alpha(k+1)(\eta+\beta-\delta))>0$. Therefore
 \begin{equation}\label{SAk}
S(\mathscr{A},(3N)^{\frac{1}{k+1}})\geq C(\eta,\beta,\delta)\frac{X}{\log (3N)}(1+o(1)),
 \end{equation}
 where $C(\eta,\beta,\delta)>0$. Theorem \ref{thm1} is derived from (\ref{HAk1}) and (\ref{SAk}).

\section{Proof of Theorem \ref{thm2}}
Here, we use the weighted sieve of Pan and Pan \cite{PP}.
We consider the weighted sum
\begin{equation}\label{W0}
  W(\mathscr{A}, k, N^{\frac{1}{15}})=\sum_{\substack{l\in\mathscr{A}\\(l,P(N^{1/15}))=1}}\left(1-\frac{1}{2}\sum_{\substack{N^{1/15}\leq p< N^{1/k}\\p|l}}1\right).
\end{equation}
We write
\begin{equation*}
  \mathscr{W}_k(l)=1-\frac{1}{2}\sum_{\substack{N^{1/15}\leq p< N^{1/k}\\p|l}}1.
\end{equation*}
On the one hand, we have
\begin{equation*}
W(\mathscr{A},k,N^{\frac{1}{15}})=
\sum_{\substack{l\in\mathscr{A}\\(l,P(N^{1/15}))=1\\ \mu^2(l)\neq 0}}\mathscr{W}_k(l)+\sum_{\substack{l\in\mathscr{A}\\(l,P(N^{1/15}))=1\\ \mu^2(l)= 0}}\mathscr{W}_k(l),
\end{equation*}
where
\begin{align*}
\sum_{\substack{l\in\mathscr{A}\\(l,P(N^{1/15}))=1\\ \mu^2(l)= 0}}\mathscr{W}_k(l)& \ll\sum_{N^{1/15}<p_1\ll N^{1/2}} \sum_{\substack{l\in\mathscr{A}\\ p_1^2|l}}1=\sum_{N^{1/15}<p_1\ll N^{1/2}}\sum_{a\in\mathfrak{A}}\sum_{b\in\mathfrak{B}}\sum_{\substack{\sqrt{ab}-\Delta<l<\sqrt{ab}+\Delta\\ p_1^2|l}}1\\
& = \sum_{N^{1/15}<p_1\ll N^{1/2}}\sum_{N^2<m\leq 4N^2}\sum_{\substack{\sqrt{m}-\Delta<l<\sqrt{m}+\Delta\\ p_1^2|l}}\sum_{a\in\mathfrak{A}}\sum_{b\in\mathfrak{B}}\mathbbm{1}_{m=ab}\\
&\ll \sum_{N^{1/15}<p_1\ll N^{1/2}}\sum_{N^2<m\leq 4N^2}\tau(m)\sum_{\substack{\sqrt{m}-\Delta<l<\sqrt{m}+\Delta\\ p_1^2|l}}1\\
& \ll N^{\varepsilon} \sum_{N^{1/15}<p_1\ll N^{1/2}}\sum_{\substack{N-\Delta<l<2N+\Delta\\ p_1^2|l}}\sum_{(l-\Delta)^2<m<(l+\Delta)^2}1\\
&\ll N^{1-\delta+\varepsilon}\sum_{N^{1/15}<p_1\ll N^{1/2}}N/p_1^2\\
&\ll N^{2-\delta-1/15+\varepsilon}.
\end{align*}
 Hence
\begin{equation}\label{W1}
W(\mathscr{A},k,N^{\frac{1}{15}})=\sum_{\substack{l\in\mathscr{A}\\(l,P(N^{1/15}))=1}}\mu^2(l)\mathscr{W}_k(l)
+O(N^{2-\delta-\varepsilon}).
\end{equation}
On the other hand, we can write
\begin{equation*}
H(\mathscr{A};k)=\sum_{l\in \mathscr{A}}\chi^{(k)}(l),
\end{equation*}
where
\begin{equation*}
 \chi^{(k)}(l)= \left\{
      \begin{array}{cll}
      1,  &&  \Omega(l)\leq k,\\
      0,  &&  \Omega(l) > k.
      \end{array}
   \right.
\end{equation*}
It is easy to see that
\begin{equation}\label{HAk2}
H(\mathscr{A};k)\geq \sum_{\substack{l\in\mathscr{A}\\(l,P(N^{1/15}))=1}}\chi^{(k)}(l)
=\sum_{\substack{l\in\mathscr{A}\\(l,P(N^{1/15}))=1}}\mu^2(l)\chi^{(k)}(l)
+O\left(N^{2-\delta-\varepsilon}\right).
\end{equation}
For each $l\in \mathscr{A}$ under the conditions $\mu^2(l)=1$ and $(l,P(N^{\frac{1}{15}}))=1$, we have:\\
$(1)$ If $\Omega(l)\leq k$, then
\begin{equation*}
\chi^{(k)}(l)=1\geq 1-\frac{1}{2}\sum_{\substack{N^{1/15}\leq p< N^{1/k}\\p|l}}1 =\mathscr{W}_k(l);
\end{equation*}
$(2)$ If $\Omega(l)\geq k+1$, then
\begin{equation*}
\sum_{\substack{N^{1/15}\leq p< N^{1/k}\\p|l}}1\geq 2,
\end{equation*}
thus
\begin{equation*}
\chi^{(k)}(l)=0\geq 1-\frac{1}{2}\sum_{\substack{N^{1/15}\leq p< N^{1/k}\\p|l}}1 =\mathscr{W}_k(l).
\end{equation*}
Therefore, by (\ref{W1}) and (\ref{HAk2}), we can get
\begin{equation}\label{HAK21}
H(\mathscr{A};k)\geq W(\mathscr{A},k,N^{\frac{1}{15}})+O(N^{2-\delta-\varepsilon}).
\end{equation}
Recalling the definition (\ref{sift}) and (\ref{W0}), we have
\begin{align}\label{W2}
  W(\mathscr{A},k, N^{\frac{1}{15}})& =
  \sum_{\substack{l\in\mathscr{A}\\(l,P(N^{1/15}))=1}}1-
  \frac{1}{2}\sum_{\substack{l\in\mathscr{A}\\(l,P(N^{1/15}))=1}}\sum_{\substack{N^{1/15}\leq p< N^{1/k}\\p|l}}1\nonumber\\
  &=S(\mathscr{A}, N^{\frac{1}{15}})-\frac{1}{2}\sum_{N^{1/15}\leq p< N^{1/k}}S(\mathscr{A}_p, N^{\frac{1}{15}}).
\end{align}
By (\ref{SAv}) with $v=15$, $\eta=\beta=1-\varepsilon$ ( $\varepsilon>0$ is sufficiently small), we get
\begin{equation*}
   S(\mathscr{A}, N^{\frac{1}{15}})\geq \frac{X}{\log N}\cdot 15e^{-\gamma}f\left(5(1-2\delta)\right) (1+o(1)).
\end{equation*}
Noticing $4<5(1-2\delta)<6$ holds for $0<\delta<\frac{1}{10}$. By Lemma \ref{lem3}, we have
\begin{align}\label{S1}
  S(\mathscr{A},N^{\frac{1}{15}})&\geq \frac{X}{\log N}\cdot\frac{6}{ 1-2\delta}(1+o(1))\left(\log(4-10\delta)+\int_{3}^{4-10\delta}\frac{1}{t}\left(\int_{2}^{t-1}\frac{\log(s-1)}{s}ds\right)dt\right)\nonumber\\
  &= \frac{X}{\log N}\cdot\frac{6}{1-2\delta}(1+o(1))\left(\log(4-10\delta)+\int_{2}^{3-10\delta}\frac{\log(s-1)}{s}\log\left(\frac{4-10\delta}{s+1}\right)ds\right).
\end{align}
Applying (\ref{upperp}) to $S(\mathscr{A}_p,N^{\frac{1}{15}})$ with $\xi^2=X^{\alpha}/p$, we can get
\begin{equation*}
S(\mathscr{A}_p,N^{\frac{1}{15}})\leq \frac{X}{p}V(N^{\frac{1}{15}})\left(F\left(\frac{\log X^{\alpha}/p}{\log N^{1/15}}\right)+o(1)\right) +\sum_{\substack{n<X^{\alpha}/p\\n|P(N^{1/15})}}3^{\nu(n)}|r(\mathscr{A},pn)|,
\end{equation*}
where $N^{\frac{1}{15}}\leq p<N^{\frac{1}{k}}$.
Hence, by (\ref{V}) and (\ref{R}), we derive
\begin{align}\label{SAp1}
\sum_{N^{1/15}\leq p< N^{1/k}}S(\mathscr{A}_p, N^{\frac{1}{15}}) \leq & \frac{X}{\log N}\cdot15e^{-\gamma}(1+o(1))\sum_{N^{1/15}\leq p<N^{1/k}}\frac{1}{p}F\left(\frac{\log X^{\alpha}/p}{\log N^{1/15}}\right)\nonumber\\
&\quad +\sum_{N^{1/15}\leq p<N^{1/k}}\sum_{\substack{n<X^{\alpha}/p\\n|P(N^{1/15})}}3^{\nu(n)}|r(\mathscr{A},pn)|\nonumber\\
\leq & \frac{X}{\log N}\cdot 15e^{-\gamma}(1+o(1))\sum_{N^{1/15}\leq p<N^{1/k}}\frac{1}{p}F\left(\frac{\log X^{\alpha}/p}{\log N^{1/15}}\right)\nonumber\\
&\quad +\sum_{d<X^{\alpha}}\mu^2(d)3^{\nu(d)}|r(\mathscr{A},d)|\nonumber\\
\leq & \frac{X}{\log N}\cdot 15e^{-\gamma}(1+o(1))\int_{N^{1/15}}^{N^{1/k}}\frac{1}{u\log u}F\left(5-10\delta-15\frac{\log u}{\log N}\right)du\nonumber\\
&\quad+O(XN^{-\frac{\varepsilon}{2}}).
\end{align}
For $N^{\frac{1}{15}}\leq u<N^{\frac{1}{k}}$, we have
\begin{equation*}
0<5-10\delta-\frac{15}{k}<5-10\delta-15\frac{\log u}{\log N}\leq 4-10\delta<4,\qquad \text{if}\quad k=4 \,\,\text{or}\,\, 5,\quad 0<\delta<\frac{1}{10}.
\end{equation*}
Then, by Lemma \ref{lem3}, we can calculate
\begin{align}\label{intF}
&15e^{-\gamma}\int_{N^{1/15}}^{N^{1/k}}\frac{1}{u\log u}F\left(5-10\delta-15\frac{\log u}{\log N}\right)du\nonumber\\
= &  15e^{-\gamma}\int_{k}^{15}\frac{1}{u}F(5-10\delta-15/u)du
=15e^{-\gamma}\int_{5-10\delta-15/k}^{4-10\delta}\frac{1}{5-10\delta-t}F(t)dt\nonumber\\
=& 30\left(\int_{5-10\delta-15/k}^{4-10\delta}\frac{1}{t(5-10\delta-t)}dt
+\int_{3}^{4-10\delta}\frac{1}{t(5-10\delta-t)}\left(\int_{2}^{t-1}\frac{\log(s-1)}{s}ds\right)dt\right)\nonumber\\
=& 30\left(\frac{1}{5-10\delta}\log\left(\frac{4-10\delta}{5-10\delta-15/k}\cdot\frac{15}{k}\right)
+\int_{2}^{3-10\delta}\frac{\log(s-1)}{s}\left(\int_{s+1}^{4-10\delta}\frac{1}{t(5-10\delta-t)}dt\right)ds\right)\nonumber\\
=& \frac{6}{1-2\delta}\left(\log\left(\frac{4-10\delta}{5-10\delta-15/k}\cdot\frac{15}{k}\right)
+\int_{2}^{3-10\delta}\frac{\log(s-1)}{s}\log\left(\frac{(4-10\delta)(5-10\delta)}{s+1}-1\right)ds\right).
\end{align}
Combining (\ref{W2})--(\ref{intF}), we can find that
\begin{align}\label{W3}
  W(\mathscr{A},k,N^{\frac{1}{15}}) & \geq \frac{X}{\log N}\frac{6}{1-2\delta}(1+o(1))\left(\log(4-10\delta)+\int_{2}^{3-10\delta}\frac{\log(s-1)}{s}\log\left(\frac{4-10\delta}{s+1}\right)ds\right.\nonumber\\
  &\quad\left.-\frac{1}{2}\log\left(\frac{4-10\delta}{5-10\delta-15/k}\cdot\frac{15}{k}\right)-\frac{1}{2}\int_{2}^{3-10\delta}\frac{\log(s-1)}{s}\log\left(\frac{(4-10\delta)(5-10\delta)}{s+1}-1\right)ds\right)\nonumber\\
  &\quad+O(XN^{-\frac{\varepsilon}{2}})\nonumber\\
  &\geq C(\delta,k)\frac{X}{\log N}(1+o(1)),
\end{align}
where
\begin{align*}
C(\delta,k)&=\frac{6}{1-2\delta}\left(\log(4-10\delta)+\int_{2}^{3-10\delta}\frac{\log(s-1)}{s}\log\left(\frac{4-10\delta}{s+1}\right)ds-\frac{1}{2}\log\left(\frac{4-10\delta}{5-10\delta-15/k}\cdot\frac{15}{k}\right)\right.\\
  &\qquad\left.-\frac{1}{2}\int_{2}^{3-10\delta}\frac{\log(s-1)}{s}\log\left(\frac{(4-10\delta)(5-10\delta)}{s+1}-1\right)ds\right).
\end{align*}
Use $Mathematica$ for numerical computation,  we can calculate:\\
(1) If $k=5$, then
\begin{equation}\label{value}
\qquad C(\delta,5)>0,\qquad \text{for} \,\,0<\delta<\frac{1}{10};
\end{equation}
(2) If $k=4$, then
\begin{equation}\label{value2}
  C(\delta,4)>0.0023205,\qquad \text{for} \,\,0<\delta<\frac{121}{10000}.
\end{equation}
From (\ref{HAK21}), (\ref{W3})--(\ref{value2}), we get the result of Theorem \ref{thm2}.

\section*{Acknowledgement}

The authors would like to express the most sincere gratitude to the referee for his/her patience and time in
refereeing this paper. This work is  partially supported  by
the National Natural Science Foundations of China (Grant Nos. 12471009, 12301006, 11901566)
and  partially supported  by  Beijing Natural Science Foundation (Grant No. 1242003).


\begin{thebibliography}{99}

\bibitem{B}O. Bordell\`es, A note on a multiplicative hybrid problem,
                                  \textit{Acta Arith.}
                                  \textbf{134} (4) (2008) 387--393


\bibitem{FI}E. Fouvry and H. Iwaniec, Exponential sums with monomials,
                                  \textit{J. Number Theory} \textbf{33} (1989) 311--333

\bibitem{GK}S. W. Graham and G. Kolesnik, Van der Corput's Method of Exponential Sums,
                                  \textit{Cambridge University Press, New York} (1991)

\bibitem{HR}H. Halberstam and H. E. Richert, Sieve Methods,
                                  \textit{Academic Press, London} (1974)

\bibitem{IS}H. Iwaniec and A. S\'ark\"ozy, On a multiplicative hybrid problem,
                                  \textit{J. Number Theory} \textbf{26} (1987) 89--95

\bibitem{PP}C. D. Pan and C. B. Pan, Goldbach Conjecture, Beijing, Science Press, 1981 (in Chinese)

\bibitem{RS}J. Rivat and A. S\'ark\"ozy, A sequences analog of the Piatetski-Shapiro problem,
                                  \textit{Acta Math. Hungar.} \textbf{74} (3) (1997) 245--260

\bibitem{Zhai}W. G. Zhai, On a multiplicative hybrid problem,
                                  \textit{Acta Arith.} \textbf{71} (1995) 47--53

\end{thebibliography}
\end{document}